\documentclass[12pt]{article}
\usepackage{graphicx}
\usepackage{amsmath, amsthm}
\usepackage{amsfonts}
\usepackage{amssymb}
\usepackage{enumitem}
\usepackage{url}
\usepackage{color,soul}
\usepackage{color}
\usepackage[dvipsnames]{xcolor}
\newtheorem{theorem}{Theorem}
\newtheorem{lemma}[theorem]{Lemma}
\newtheorem{cor}[theorem]{Corollary}
\newtheorem{prop}[theorem]{Proposition}

\newtheorem*{reduc1}{Operation I}
\newtheorem*{reduc2}{Operation II}

\newtheorem{example}{Example}

\newtheorem*{thm:Tpath}{Theorem \ref{thm:Tpath}}
\newtheorem*{thm:Cpath}{Theorem \ref{thm:Cpath}}

\theoremstyle{definition}

\newcommand{\mc}{\mathcal}

\title{Hamilton Paths in Dominating Graphs of Trees and Cycles}
\author{Kira Adaricheva \thanks{Department of Mathematics, Hofstra University,
Hempstead, NY 11549, USA (kira.adaricheva@hofstra.edu)}
\and Heather Smith Blake \thanks{Department of Mathematics and Computer Science,
Davidson College, Davidson, NC 28035, USA (hsblake@davidson.edu)}
\and Chassidy Bozeman \thanks{Department of Mathematics and Statistics,
Mount Holyoke College, South Hadley, MA 01075, USA (cbozeman@mtholyoke.edu)}
\and Nancy E. Clarke \thanks{Department of Mathematics and Statistics, Acadia University,
Wolfville, NS B4P 2R6, Canada (nancy.clarke@acadiau.ca)}
\and Ruth Haas \thanks{Department of Mathematics, University of Hawaii at
M\=anoa, Honolulu, HI 96822, USA (rhaas@hawaii.edu)}
\and Margaret-Ellen Messinger \thanks{Department of Mathematics and Computer Science,
Mount Allison University, Sackville, NB E4L 1E2, Canada (mmessinger@mta.ca)}
\and Karen Seyffarth \thanks{Department of Mathematics and Statistics, University
of Calgary, Calgary, Alberta T2N 1N4, Canada (kseyffar@ucalgary.ca)}}

\begin{document}
\maketitle

\begin{abstract}
The dominating graph of a graph $H$ has as its vertices all dominating sets of $H$,
with an edge between two dominating sets if one can be obtained from the other by
the addition or deletion of a single vertex of $H$.
In this paper we prove that the dominating graph of any tree has a Hamilton path. 
We also show how a result about binary strings leads to a proof that
the dominating graph of a cycle on $n$ vertices has a Hamilton path 
if and only if $n\not\equiv 0 \pmod 4$. 
\end{abstract}

%%%%%%%%%%%%%%%%%%%%%%%%%%%%%%%%%%%%%%%%%%%%%%%%%%%%%%%%%%%%%%%%%%%%%
\section{Introduction}
%%%%%%%%%%%%%%%%%%%%%%%%%%%%%%%%%%%%%%%%%%%%%%%%%%%%%%%%%%%%%%%%%%%%%

Let $H$ be a graph with vertex set $V(H)$.
A {\it dominating set} of $H$ is a set $D \subseteq V(H)$ such 
that every vertex of $V(H)\backslash D$ is adjacent to a vertex of $D$.  
The {\em dominating graph} of $H$, $\mc{D}(H)$, is the graph whose vertices are all the 
dominating sets of $H$; 
if $X$ and $Y$ are distinct vertices of $\mc{D}(H)$, then there is an edge between $X$
and $Y$ if and only if $Y$ can be obtained from $X$ by adding a vertex of $H$ to $X$ or 
by deleting a vertex from $X$.
Note that we use the same label for a vertex of $\mc{D}(H)$ as for the corresponding 
dominating set of $H$ because it is clear from context whether we are 
referring to $H$ or $\mc{D}(H)$.

The graph $\mc{D}(H)$ is the reconfiguration graph of dominating sets of $H$ under the 
{\em token addition/removal} (TAR) model, first considered in~\cite{HS14}.
For any graph $H$ and any integer $k$, $1\leq k\leq |V(H)|$, the {\em $k$-dominating graph} 
of $H$, denoted $\mc{D}_k(H)$, is the subgraph of $\mc{D}(H)$ induced by the dominating 
sets of $H$ with 
cardinality at most $k$.
When $k=|V(H)|$, then $\mc{D}_k(H)=\mc{D}(H)$.
There have been numerous papers about dominating graphs and their subgraphs, 
the $k$-dominating graphs.
Most of these focus on conditions on $k$ that ensure that $\mathcal{D}_k(H)$ is connected.  
Two recent surveys of reconfiguration of dominating sets are~\cite{us} and~\cite{ChungBook}.

There has been considerable interest in reconfiguration and reconfiguration graphs of other 
well known graph structures and operations, including independent sets, cliques,
vertex covers of graphs, zero forcing, and graph coloring.  
Nishimura \cite{Nishimura-survey} examines reconfiguration from an algorithmic 
perspective and considers complexity questions in a wide range of 
reconfiguration settings.
Reconfiguration of graph coloring problems and dominating set problems are surveyed
in a recent paper of Mynhardt and Nasserasr \cite{ChungBook}. 

In this paper we investigate Hamilton cycles and Hamilton paths in dominating graphs, 
properties that have been studied for other types of reconfiguration problems. 
A Hamilton path or Hamilton cycle in a reconfiguration graph is a 
{\em combinatorial Gray code}, that is, a listing of all the objects in a set so
that successive objects differ in some prescribed minimal way. 
Several recent papers give conditions for the existence of Gray codes for all
colorings with $k$ or fewer colors of the following classes of graphs:
trees~\cite{choo}, bipartite graphs~\cite{Celaya}, and 2-trees~\cite{cavers}.
A forthcoming survey by M\"{u}tze gives a wide variety of combinatorial Gray code
results~\cite{mutze}.

We consider only finite simple graphs.
For a graph $H$, we use the notation $P=x_1, x_2, \ldots, x_j$, where $j\geq 3$,
to denote a path $P$ in $H$ where $\{x_1, x_2, \ldots, x_j\}$ is a subset of the vertices of $H$.
An edge $e$, i.e., a path with two vertices $x$ and $y$, is written simply as $e=xy$.
For basic graph theory notation and terminology not defined here, see~\cite{BondyandMurty}. 

We begin with the question of which dominating graphs have Hamilton cycles.
It is clear that if $H$ is a graph, then its dominating graph, $\mc{D}(H)$, is bipartite,
with the bipartition based on the parity of the dominating sets of $H$. 
It follows that if $\mc{D}(H)$ has a Hamilton cycle, then $\mc{D}(H)$ has an even number 
of vertices (equivalently, the number of dominating sets of $H$ is even).
By contrast, we have the following unpublished result of Brouwer~\cite{BCS},
an expanded proof of which is included in~\cite{us}.

\begin{lemma}\cite{BCS}\label{lem:Brouwer} 
The number of dominating sets of any finite graph is odd.
\end{lemma}

Combining Brouwer's result with the observation that a bipartite graph
with a Hamilton cycle must have an even number of vertices gives a short
answer to the question of which dominating graphs have Hamilton cycles.

\begin{prop}\cite{us}
For any graph $H$, the dominating graph ${\mc D}(H)$ has no Hamilton cycle.
\end{prop}

\noindent
Henceforth, we focus our attention on Hamilton paths in dominating graphs.  

In \cite{us} we show that Hamilton paths exist in the dominating graphs of
certain classes of graphs.
Specifically, we prove the following.

\begin{theorem}\cite{us}\label{thm:oldresults}
Let $m$ and $n$ be positive integers.
Then $\mc{D}(K_n)$ has a Hamilton path, $\mc{D}(P_n)$ has a Hamilton path, and 
$\mc{D}(K_{m,n})$ has a Hamilton path if and only if $m$ is odd.
\end{theorem}

\noindent
In this paper we explore the dominating graphs of trees and prove the following.

\begin{theorem}\label{thm:Tpath}
For any tree $T$, ${\mc D}(T)$ has a Hamilton path. 
\end{theorem} 

\noindent
We also use a result of Baril and Vajnovszki~\cite{Baril} on Lucas strings to
characterize cycles whose dominating graphs have Hamilton paths.

\begin{theorem}\label{thm:Cpath}%\cite{us-cycle, Baril}
For all integers $n\geq 3$, $\mc{D}(C_n)$ has a Hamilton path if and 
only if $n\not\equiv 0 \pmod 4$.
\end{theorem}

%%%%%%%%%%%%%%%%%%%%%%%%%%%%%%%%%%%%%%%%%%%%%%%%%%%%%%%%%%%%%%%%%%%%%
\section{Two graph operations and their effects on the dominating graph}\label{sec:graph-ops}
%%%%%%%%%%%%%%%%%%%%%%%%%%%%%%%%%%%%%%%%%%%%%%%%%%%%%%%%%%%%%%%%%%%%%

We begin this section by introducing two operations on a graph $H$.
We then prove that if $H'$ is a graph obtained 
from $H$ by applying either operation, and $\mathcal{D}(H')$ has a Hamilton path,
then $\mathcal{D}(H)$ has a Hamilton path. 
This is later used to show that the dominating graph of any tree has a Hamilton path. 

\begin{reduc1}
Let $H$ be a graph with vertices $u, v$ and $x$
such that $N_{H}(u)= N_{H}(v)=\{x\}$.
We say that $H':=H-v$ is obtained from $H$ by Operation I. 
\end{reduc1}

\begin{figure}[htbp]
\begin{center}
\includegraphics[width=0.8\textwidth]{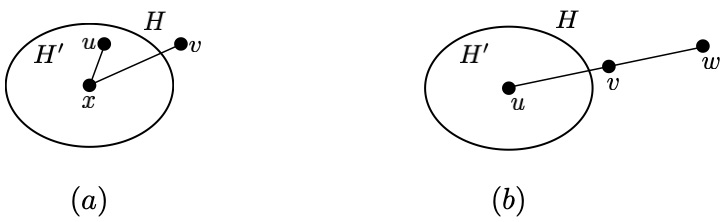}
\end{center}
\caption{(a) Operation I and (b) Operation II}
\label{fig:reductions}
\end{figure}

\begin{reduc2}
Let $H$ be a graph with vertices $u, v$ and $w$ such that
$N_{H}(v)=\{u, w\}$ and $N_{H}(w)=\{v\}$.
We say that $H':=H-w-v$ is obtained from $H$ by Operation II.
\end{reduc2}

\begin{lemma}\label{thm:reduction1}\rm
Let $H$ and $H'$ be graphs such that $H'$ is obtained from $H$ by Operation I. 
If $\mathcal{D}(H')$ has a Hamilton path,
then $\mathcal{D}(H)$ has a Hamilton path.
\end{lemma}

\begin{proof}
Suppose $H$ and $H'$ are graphs as in the statement of the Proposition.
Then there are vertices $u,v,x\in V(H)$ such that $N_{H}(u)=N_{H}(v)=\{ x\}$,
and $H':=H-v$.
To simplify notation, we define $G$ and $G'$ to be the dominating graphs of
$H$ and $H'$, respectively, i.e.,
$G:= {\mathcal D}(H)$ and $G':= {\mathcal D}(H')$.
Recall that each vertex of $G$ represents a dominating set of $H$, so we name
each vertex of $G$ with the name of the corresponding subset of $V(H)$,
and use the same convention for vertices of $G'$ and dominating sets of $H'$.

Let $n=|V(G')|$, and
let $P_{G'}:=F_1, F_2, \ldots, F_n$ be a Hamilton path in $G'$.
For each $i$, $1\leq i \leq n$, define $F^v_i := F_i\cup \{v\}$,
and for each $i$, $1\leq i \leq n$, with $u\not\in F_i$, define
\[ F^u_i := F_i\cup \{u\}, \text{ and } F^{uv}_i := F_i\cup \{u,v\}.  \]
Now consider the dominating sets of $H$.
These can be partitioned into those that contain $v$ and those that do
not contain $v$.
Because $N_{H}(u)=\{x\}=N_{H}(v)$,
the dominating sets of $H$ that contain $v$ are precisely
\[ W:=\{F_i^v ~|~ 1\leq i \leq n\},\]
while the dominating sets of $H$ that do not
contain $v$ are
\[ Z:=\{F_i ~|~ x\in F_i, 1\leq i\leq n\}.\]
Now consider the following subsets of $V(G)$.
\begin{eqnarray*}
X' &:=& \{F_i ~|~ x\in F_i, \, u\not\in F_i, \,  1\leq i\leq n\},\\
B' &:=& \{F_i ~|~ \{x,u\}\subseteq F_i, \, 1\leq i\leq n\}
= \{F_i^u ~|~ F_i \in X'\},\\
X &:=& \{F_i^v ~|~ x\in F_i, \, u\not\in F_i, \, 1\leq i\leq n\}
= \{F_i^v  ~|~ F_i \in X'\}, \\
B &:=& \{F_i^v ~|~ \{x,u\}\subseteq F_i, \, 1\leq i\leq n\}
= \{F_i^{uv} ~|~ F_i \in X'\}, \\
U &:=& \{F_i^v ~|~ u\in F_i, \, x\not\in F_i, \, 1\leq i\leq n\}.
\end{eqnarray*}

\begin{figure}[htbp]
\[ \includegraphics[width=0.45\textwidth]{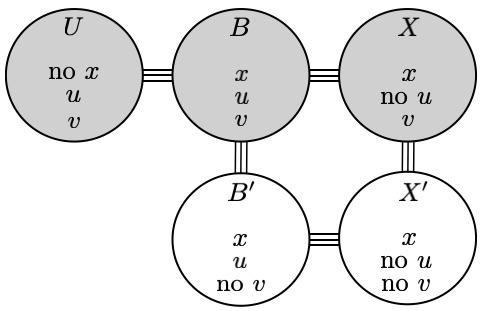}\]
\caption{A partition of $V(G)=V(\mathcal{D}(H))$ into sets
$\{U,B,X,B',X'\}$, with the parts shaded grey corresponding to
a subgraph of $G$ isomorphic to 
$G'=\mathcal{D}(H')$.}\label{Fig_ab}
\end{figure}

It is routine to verify that $\{X',B'\}$ is a partition of $Z$ while
$\{ X, B, U\}$ is a partition of $W$,
and hence
$\{U, B, X, B', X' \}$ is a partition of $V(G)$.
Furthermore, the definitions of $B', X$ and $B$ in terms of $X'$ make it clear that
\[ G[B']\cong G[X]\cong G[B] \cong G[X'].\]
Finally, the definition of $W$ ensures that $G[W]=G[X\cup B\cup U]\cong G'$.
It follows that $P:=F_1^v, F_2^v, \ldots, F_n^v$ is a path in $G$ and
also a Hamilton path of
$G[X \cup B \cup U]$.
We now extend $P$ to a Hamilton path of $G$.

Let $Q:=F_i^v, F_{i+1}^v, \ldots, F_j^v$ be a maximal subpath of $P$ in $G[X]$.
There are two cases to consider, depending on the parity of $j-i+1$ (the
number of vertices in $Q$).
First suppose that $j-i+1$ is even.
Then for each $t\in \{i, i+2, \ldots, j-1\}$, replace the edge $F_t^v F_{t+1}^v$
of $P$ by the path
\[ F_t^v, F_t, F_t^u, F_{t+1}^u, F_{t+1}, F_{t+1}^v.\]
Since $G[B']\cong G[X]$, this replacement results in a path in $G$.

Now assume that  $j-i+1$ is odd.
In this case, for each $t\in \{i,i+2, \ldots, j-2\}$,
replace the edge $F_t^v F_{t+1}^v$ of $P$ by the path
\[ F_t^v, F_t, F_t^u, F_{t+1}^u, F_{t+1}, F_{t+1}^v.\]
Again, since $G[B']\cong G[X]$, this replacement results in a path in $G'$.
If $j=n$, replacing vertex $F_n^v$ in $P$ by the path $F_n^v, F_n, F_n^u$ results in a path.
Otherwise, $j<n$, so the maximality of $Q$ implies $F_{j+1}^v \in B$, and hence
$F_{j+1}^v = F_j^{uv}$.
Replacing the edge $F_j^v F_{j+1}^v$ (which equals $F_j^v  F_j^{uv}$) of $P$
with the path
\[ F_j^v,F_j,F_j^u,F_j^{uv}\]
ensures the result is a path in $G$.

Since $G[B']\cong G[X']\cong G[X]$,
making these replacements for each maximal subpath $Q$ of $P$ in $G[X]$ incorporates
all the vertices of $X'$ and $B'$ into the resulting path and produces a 
Hamilton path of $G=\mc{D}(H)$.
\end{proof}

\begin{lemma}\label{thm:reduction2}\rm
Let $H$ and $H'$ be graphs such that $H'$ is obtained from $H$ by Operation II. 
If $\mathcal{D}(H')$ has a Hamilton path, then $\mathcal{D}(H)$ has a Hamilton path.
\end{lemma}

\begin{proof} 
Suppose $H$ and $H'$ are graphs as in the statement of the Proposition.
Then there exist vertices $u,v,w\in V(H)$ such that $N_{H}(v)=\{u,w\}$,
$N_{H}(w)=\{v\}$, and $H':=H-w-v$.
As before, we define $G$ and $G'$ to be the dominating graphs of
$H$ and $H'$, respectively, i.e.,
$G:= {\mathcal D}(H)$ and $G':= {\mathcal D}(H')$.

Let $Y$ be a dominating set of $H'$.
Since neither $v$ nor $w$ is a vertex of $H'$, $Y\cap\{v,w\}=\emptyset$,
so we define
\[ Y^v:=Y\cup\{v\}, Y^w:=Y\cup\{w\},\mbox{ and } Y^{vw}:=Y\cup\{v,w\},\]
and let
\[ A:=\{Y^v, Y^w, Y^{vw} ~|~ Y\in V(G')\}\]
Then $A$ consists of dominating sets of $H$.
The dominating sets of $H$ that are {\em not} in $A$ can be described as
follows.
Let
\[ J:= \{ S\subseteq V(H') ~|~ S \mbox{ is a dominating set of } H'-u
\mbox{ and } S\cap N_{H'}[u] = \emptyset\},\]
i.e., $J$ consists of the dominating sets of $H'-u$ that are {\em not} dominating
sets of $H'$.  
It follows that if $S\in J$, then $S\cap\{u,v,w\}=\emptyset$,
so we define
\begin{eqnarray*}
& S^u:=S\cup\{u\}, S^v:=S\cup\{v\}, S^{uv}:=S\cup\{u,v\}, & \\
& S^{vw}:=S\cup\{v,w\},\mbox{ and } S^{uvw}:=S\cup\{u,v,w\}.
\end{eqnarray*}
We now let
\[ B:=\{S^v, S^{vw} ~|~ S \in J\}.\]
It is routine to verify that $\{ A,B\}$ is a partition of the dominating sets of $H$.

Let $n=|V(G')|$, and let $P_{G'}=F_1, F_2, \ldots, F_n$ be a Hamilton path in $G'$.
Note that $F_i\cap\{v,w\}=\emptyset$.
By Lemma~\ref{lem:Brouwer}, $n$ is odd, so replacing vertex $F_i$ of 
$P_{G'}$ with the path 
$F_i^v, F_i^{vw}, F_i^w$ when $i$ is odd,  
and with the path $F_i^w, F_i^{vw}, F_i^v$ when $i$ is even produces the path
\[ P := F_1^v, F_1^{vw}, F_1^w, F_2^w, F_2^{vw}, F_2^v, \ldots , F_n^v, F_n^{vw}, F_n^w \]
in $G$.
Since $P$ consists of all the vertices in $A$, 
what remains is to incorporate the vertices of $B$ into this path.

First notice that, for each $S\in J$, $S^u$ is a dominating set of $H'$,
and hence $S^u=F_i$ for some $i$, $1\leq i\leq n$.
Furthermore, it is clear that if $S_1, S_2\in J$, then $S_1\neq S_2$ if and only if
$S_1^u\neq S_2^u$.

We now proceed as follows.  
For $S\in J$, let $t\in \{ 1, \ldots, n\}$ be the index for which
$S^u=F_t$.
In the path $P$, we either have the edge $F^v_t F^{vw}_t$
(which is the same as $S^{uv}S^{uvw}$) or the edge $F^{vw}_t F^{v}_t$
(which is the same as $S^{uvw}S^{uv}$).
If $P$ contains $F^v_t F^{vw}_t$, replace it with the path
\[F^v_t=S^{uv}, S^v, S^{vw}, S^{uvw}=F^{vw}_t.\]
Otherwise, replace $F^{vw}_t F^{v}_t$ with the path
\[ F^{vw}_t=S^{uvw}, S^{vw}, S^v, S^{uv}=F^{v}_t.\]
Repeating this for each $S\in J$ results in a 
path containing all the vertices of $A\cup B=V(G)$, and hence all 
the dominating sets of $H$.
Therefore, $G=\mc{D}(H)$ has a Hamilton path. 
\end{proof}

Together, the two preceding propositions imply the following.

\begin{cor}\label{cor:successive}
Let $H$ be a graph and let $H'$ be a graph obtained from $H$ by 
applying any sequence of the Operations I and II.
If $\mathcal{D}(H')$ has a Hamilton path then $\mathcal{D}(H)$ has a Hamilton path.
\end{cor}

%%%%%%%%%%%%%%%%%%%%%%%%%%%%%%%%%%%%%%%%%%%%%%%%%%%%%%%%%%%%%%%%%%%%%
\section{Hamilton paths in dominating graphs of trees}\label{sec:trees}
%%%%%%%%%%%%%%%%%%%%%%%%%%%%%%%%%%%%%%%%%%%%%%%%%%%%%%%%%%%%%%%%%%%%%

A particular class of graphs to which we can apply Corollary~\ref{cor:successive}
is trees.
Let $T$ be a tree and let ${\mathcal D}(T)$ be the dominating graph of $T$.  
To prove that ${\mathcal D}(T)$ has a Hamilton path (Theorem~\ref{thm:Tpath}), we
use an iterative process for constructing such a path.
Doing so requires the following lemma to deconstruct an arbitrary tree on $n\geq 3$ 
vertices using Operations~I and~II.

\begin{lemma}\rm\label{thm:reductions}
If $T$ is a tree on $n\geq 3$ vertices, then one of the following holds: 
\begin{enumerate}[label={(\arabic*)}]
\item there exist distinct $u,v,x\in V(T)$ with $N_{T}(u)=N_{T}(v)=\{x\}$, or
\item there exist distinct $u,v,w\in V(T)$ with $N_{T}(v)=\{u,w\}$ and 
$N_{T}(w)=\{v\}$.
\end{enumerate}
\end{lemma}

\begin{proof}
The proof is by induction on $n$.
When $n=3$, then $T\cong P_3$ and the result is obvious.

Suppose $n\geq 4$.
Let $z\in V(T)$ be a vertex of degree one, and let $T'=T-z$.
By the induction hypothesis, $T'$ satisfies (1) or (2) of the
statement of the Lemma.

First suppose that $T'$ satisfies (1), and let $u,v,x\in V(T')$ with 
$N_{T'}(u)=N_{T'}(v)=\{x\}$.
If $N_T(z)\not\subseteq\{u,v\}$, then $N_{T}(u)=N_{T}(v)=\{x\}$,
and $T$ satisfies (1).
Otherwise, $N_T(z)\subseteq\{u,v\}$ and we may assume, without loss of generality
that, $N_T(z)=\{u\}$.
We now have $z, u, x\in V(T)$ with $N_{T}(u)=\{z,w\}$ and 
$N_{T}(z)=\{u\}$, so $T$ satisfies (2).

Now suppose that $T'$ satisfies (2), and let 
$u,v,w\in V(T')$ with $N_{T'}(v)=\{u,w\}$ and $N_{T}(w)=\{v\}$.
If $N_T(z)\not\subseteq\{v,w\}$, then $N_{T}(v)=\{u,w\}$ and $N_{T}(w)=\{v\}$,
so $T$ satisfies (2).
Otherwise, $N_T(z)=\{v\}$ or $N_T(z)=\{w\}$.
There are two cases to consider.
If $N_T(z)=\{v\}$, then $z,v,w\in V(T)$ and 
$N_{T}(w)=N_{T}(z)=\{v\}$, and thus $T$ satisfies (1).
If $N_T(z)=\{w\}$, then $z,v,w\in V(T)$, $N_{T}(w)=\{v,z\}$ and 
$N_{T}(z)=\{w\}$, so $T$ satisfies (2).
\end{proof}
We are now in a position to prove our main result about trees, first stated in the 
Introduction.
\begin{thm:Tpath}
For any tree $T$, $\mathcal{D}(T)$ has a Hamilton path.
\end{thm:Tpath}

\begin{proof}
Let $P_i$ denote the path on $i\geq 1$ vertices.
If $|V(T)|\leq 2$, then $T\cong P_1$ or $T\cong P_2$.
Since $\mc{D}(P_1)\cong P_1$ and $\mc{D}(P_2)\cong P_3$, 
$\mathcal{D}(T)$ has a Hamilton path.
If $|V(T)|\geq 3$, then by Lemma~\ref{thm:reductions}, we can repeatedly apply 
Operations I and II to $T$ to obtain a tree $T'$ with $|V(T')|\leq 2$. 
Since $\mathcal{D}(T')$ has a Hamilton path,
it follows from Corollary~\ref{cor:successive} that
$\mathcal{D}(T)$ has a Hamilton path.
\end{proof}

%%%%%%%%%%%%%%%%%%%%%%%%%%%%%%%%%%%%%%%%%%%%%%%%%%%%%%%%%%%%%%%%%%%%%
\section{Hamilton paths in dominating graphs of cycles}\label{sec:cycles}
%%%%%%%%%%%%%%%%%%%%%%%%%%%%%%%%%%%%%%%%%%%%%%%%%%%%%%%%%%%%%%%%%%%%%

Let $C_n$ denote the cycle on $n\geq 3$ vertices.
In our original construction of a Hamilton path in $\mathcal{D}(C_n)$
if and only if $n\not\equiv 0\pmod 4$, we  encode
dominating sets of $C_n$ as binary strings, 
and construct a Gray code of this set of strings.
It was pointed out to us by T.\ M\"utze that the set of strings corresponding to the dominating sets of $C_n$ are the bitwise complements of the Lucas strings $L_{n,3}$.
Further, the Gray codes of Lucas strings are well-understood,
and thus we use them for the proof presented here.

For our purposes, the cycle on $n\geq 3$ vertices has vertex set 
$V(C_n)=\{ 0, 1, \ldots, n-1\}$ and edge set $\{ ij : i-j\equiv \pm 1\pmod n\}$. 
We encode  $X\subseteq V(C_n)$ as an $n$-digit binary string,
$x_0x_1\cdots x_{n-1}$, by setting $x_i=1$ if and only if $i\in X$, $0\leq i\leq n-1$.
It follows that $X\subseteq V(C_n)$, encoded by the binary string $x_0 x_1 \cdots x_{n-1}$,
is a dominating set of $C_n$ if and only if $x_{i-1} x_i x_{i+1} \neq 000$
for all $i$, $0\leq i\leq n-1$, where subscripts are taken modulo $n$. 
If $X$ and $Y$ are dominating sets of $C_n$, and are represented by binary 
strings $x_0 x_1 \cdots x_{n-1}$ and $y_0 y_1 \cdots y_{n-1}$,
respectively, then $X$ and $Y$ are adjacent in $\mc{D}(C_n)$ if and only if
$x_0 x_1 \cdots x_{n-1}$ and $y_0 y_1 \cdots y_{n-1}$ differ in exactly one bit.

The set of {\em Lucas strings of length $n$ and order $p\geq 1$},
denoted $L_{n,p}$, is the set of binary strings
of length $n$ that have no $p$ consecutive ones when the strings
are considered circularly.
In particular, the set  of Lucas strings of length $n$ and order $3$ is
\[ L_{n,3}=\{ x_0\cdots x_{n-1} ~|~ x_{i-1} x_i x_{i+1} \neq 111
\mbox{ for } 0\leq i\leq n-1 \mbox{ subscripts mod }n \}, \]
and is the set of bitwise complements of elements of $V(C_n)$.

Baril and Vajnovszki~\cite{Baril} construct an ordering of the
elements of $L_{n,p}$ called a {\em minimal change list} (see~\cite{Baril}),
denoted by $\mc{L}_{n,p}$. 
They prove $\mc{L}_{n,p}$ is a Gray code if and only if $n\not\equiv 0\pmod{(p+1)}$;
that is, every pair of consecutive strings of $\mc{L}_{n,p}$ differs
in exactly one bit.
Let $\widehat{\mc{L}}_{n,p}$ denote the sequence obtained by taking
bitwise complements of the strings of $\mc{L}_{n,p}$,
and note that $\mc{L}_{n,p}$ is a Gray code if and only if
$\widehat{\mc{L}}_{n,p}$ is a Gray code.
Since a Gray code of $\widehat{\mc{L}}_{n,3}$ corresponds
precisely to a Hamilton path in $\mc{D}(C_n)$,
this proves the following.

\begin{thm:Cpath}
For all integers $n\geq 3$, $\mc{D}(C_n)$ has a Hamilton path if and
only if $n\not\equiv 0 \pmod 4$.
\end{thm:Cpath}

A computationally inefficient construction of $\mc{L}_{n,p}$ 
(though not the construction used in the proof)
is described in~\cite{Baril}, and can easily be 
modified to directly construct a Hamilton 
path of $\mc{D}(C_n)$ whenever $n\not\equiv 0 \pmod 4$.
We illustrate this construction in Example~\ref{modified-lucas}.

\begin{example}\label{modified-lucas}\rm
Let $n=5$.
To construct a Hamilton path of $\mc{D}(C_5)$, begin with the
{\em reflected Gray code order} (due to Frank Gray~\cite{Gray})
of the set of all binary strings of length five.
The strings are organized in Figure~\ref{fig:gray}(a) to be read
from top to bottom and left to right.
Next, delete any string $x_0x_1x_2x_3x_4$ that has
$x_{i-1} x_i x_{i+1} =000$ for $0\leq i\leq 4$,
subscripts modulo $4$.
\begin{figure}
{\footnotesize
\begin{tabular}{|rrrrrrr|}
\hline
00000 &~& 01100 &~& 11000 &~& 10100 \\
00001 &~& 01101 &~& 11001 &~& 10101 \\
00011 &~& 01111 &~& 11011 &~& 10111 \\
00010 &~& 01110 &~& 11010 &~& 10110 \\
00110 &~& 01010 &~& 11110 &~& 10010 \\
00111 &~& 01011 &~& 11111 &~& 10011 \\
00101 &~& 01001 &~& 11101 &~& 10001 \\
00100 &~& 01000 &~& 11100 &~& 10000 \\ \hline
\multicolumn{7}{c}{(a)}
\end{tabular}
\hspace*{.25in}
\begin{tabular}{|rrrrrrr|}
\hline
      &~&       &~&       &~& 10100 \\
      &~& 01101 &~& 11001 &~& 10101 \\
      &~& 01111 &~& 11011 &~& 10111 \\
      &~& 01110 &~& 11010 &~& 10110 \\
      &~& 01010 &~& 11110 &~& 10010 \\
00111 &~& 01011 &~& 11111 &~& 10011 \\
00101 &~& 01001 &~& 11101 &~&       \\
      &~&       &~& 11100 &~&       \\ \hline
\multicolumn{7}{c}{(b)}
\end{tabular}}
\caption{Constructing a Hamilton path in $\mc{D}(C_5)$.}
\label{fig:gray}
\end{figure}
The reader can easily verify that remaining strings,
shown in Figure~\ref{fig:gray}(b), are still a Gray code
when read from top to bottom and left to right, and hence
describe a Hamilton path in $\mc{D}(C_5)$.
\end{example}

%%%%%%%%%%%%%%%%%%%%%%%%%%%%%%%%%%%%%%%%%%%%%%%%%%%%%%%%%%%%%%%%%%
\section{Further results}
%%%%%%%%%%%%%%%%%%%%%%%%%%%%%%%%%%%%%%%%%%%%%%%%%%%%%%%%%%%%%%%%%%

Corollary~\ref{cor:successive} applies more generally and can be used to
prove the existence of Hamilton paths in classes of dominating graphs
that are built up using dominating graphs that are known to have Hamilton paths.
These include complete graphs, paths, cycles $C_n$ when $n\not\equiv 0\pmod 4$,
certain complete bipartite graphs (Theorem~\ref{thm:oldresults}), and 
trees (Theorem~\ref{thm:Tpath}).
We include one example.

For any graph $H$, we say that $H$ is {\em reducible} to subgraph $H'$ if
$H'$ can be obtained from $H$ by applying a sequence of Operations~I and II
as described in Section~\ref{sec:graph-ops}.  
Suppose $G$ is a unicyclic graph whose unique cycle $C_n$ has length $n\geq 3$, 
where $n \not\equiv 0 \pmod 4$. 
Let $V(C_n)=\{ v_1, v_2, \ldots, v_n\}$, and let $T_i$ be the component (a tree) of 
$G-E(C_n)$ containing $v_i$ for some $i$, $1\leq i\leq n$.  
If $T_i$ is reducible to $v_i$ for each $i$, $1\leq i \leq n$, then by
Theorem~\ref{thm:Cpath} and Corollary~\ref{cor:successive}, 
$\mc{D}(G)$ has a Hamilton path. 

%%%%%%%%%%%%%%%%%%%%%%%%%%%%%%%%%%%%%%%%%%%%%%%%%%%%%%%%%%%%%%%%%%
\section*{Acknowledgements}
%%%%%%%%%%%%%%%%%%%%%%%%%%%%%%%%%%%%%%%%%%%%%%%%%%%%%%%%%%%%%%%%%%

This project started at the {\em Women in Graph Theory and its Applications
(WIGA) Workshop} at the Institute for Mathematics and its Applications in 2019.
WIGA is a research network  of the AWM, with support provided by the 
AWM ADVANCE Research Communities Program (funded by NSF-HDR-1500481, Career 
Advancement for Women Through Research-Focused Networks).
We gratefully acknowledge the support of the AWM, NSF and IMA.

We also thank Torsten M\"{u}tze for bringing our attention to the work of
Baril and Vajnovszki~\cite{Baril}.


\begin{thebibliography}{99}

\bibitem{us} K.\ Adaricheva, C.\ Bozeman, N.E.\ Clarke, R.\ Haas,
M.E.\ Messinger, K.\ Seyffarth and H.\ Smith,
Reconfiguration graphs for dominating sets, Chapter 6 in
{\em Research Trends in Graph Theory and Applications}, D.\ Ferrero, L.\ Hogben, 
S.R.\ Kingan and G.L.\ Matthews Eds., Springer International Publishing (2021), 
DOI https://doi.org/10.1007/978-3-030-77983-2.

\bibitem{Baril} J-L.\ Baril and  V.\ Vajnovszki, 
Minimal change list for Lucas strings and some graph theoretic consequences.
Theoretical Computer Science 346 (2005) 189--199.

\bibitem{BondyandMurty} J.A.\ Bondy and U.S.R.\ Murty, {\it Graph Theory}.
GTM 244, Springer, Berlin (2008).

\bibitem{BCS} A.\ Brouwer, P.\ Csorba, and L.\ Schrijver,
The number of dominating sets of a finite graph is odd, preprint (2009)
\url{https://www.win.tue.nl/~aeb/preprints/domin4a.pdf}.

\bibitem{cavers} M.\ Cavers and K.\ Seyffarth,
Reconfiguring vertex colourings of 2-trees,
Ars Math.\ Contemp.\ 17(2) (2019) 653--698. 

\bibitem{Celaya} M.\ Celaya, K.\ Choo, G.\ MacGillivray and K.\ Seyffarth,
Reconfiguring $k$-colourings of complete bipartite graphs,
Kyungpook Math.\ J.\ 56 (2016) 647--655.

\bibitem{choo} K.\ Choo and G.\ MacGillivray,
Gray code numbers for graphs,
Ars Math. Contemp. 4(1) (2011) 125--139.

\bibitem{Gray} F.\ Gray, Pulse code communication, US Patent 2632058, 1953.

\bibitem{HS14} R.\ Haas and K.\ Seyffarth,
The $k$-dominating graph, Graphs and Combin. 30 (2014) 609--617. 

\bibitem{Nishimura-survey} N.\ Nishimura,
Introduction to reconfiguration, Algorithms 11(4), 52 (2018).

\bibitem{ChungBook} C.M.\ Mynhardt and S.\ Nasserasr,
Reconfiguration of colourings and dominating sets in graphs (2020).
In: F.\ Chung, R.\ Graham, F.\ Hoffman, L.\ Hogben, R.C.\ Mullin and D.B.\ West, eds.\
{\it 50 years of combinatorics, graph theory, and computing}, Florida, CRC Press, 171--187.

\bibitem{mutze} T.\ M\"{u}tze, Combinatorial Gray codes -- an updated survey, in preparation. 
\end{thebibliography}
\end{document}